\documentclass[numbook,envcountsame,smallcondensed]{amsart}

\usepackage{amssymb,amsmath,amsfonts,mathrsfs,cite,mathptmx,graphicx,gastex,url}

\DeclareMathOperator{\alf}{alph}
\DeclareMathOperator{\simple}{sim}
\DeclareMathOperator{\mul}{mul}
\DeclareMathOperator{\occ}{occ}

\newtheorem{theorem}{Theorem}
\newtheorem{proposition}[theorem]{Proposition}
\newtheorem{lemma}[theorem]{Lemma}

\newtheorem{claim}{Claim}

\newcommand{\Mon}{\mathbb{M}\text{\scriptsize$\mathbb{ON}$}}
\newcommand{\Sem}{\mathbb{S}\text{\scriptsize$\mathbb{EM}$}}
\newcommand{\vMon}{\mathbf{M}\text{\scriptsize$\mathbf{ON}$}}

\makeatletter

\renewcommand*\subjclass[2][2010]{\def\@subjclass{#2}\@ifundefined{subjclassname@#1}{\ClassWarning{\@classname}{Unknown edition (#1) of Mathematics Subject Classification; using '2010'.}}{\@xp\let\@xp\subjclassname\csname subjclassname@#1\endcsname}}

\renewcommand{\subjclassname}{\textup{2010} Mathematics Subject Classification}

\makeatother

\begin{document}

\title{Modular elements in the lattice of monoid varieties}
\thanks{The research was supported by the Ural Mathematical Center, Project No. 075-02-2025-1719/1.}

\author{Sergey V. Gusev}

\address{Ural Federal University, Institute of Natural Sciences and Mathematics, Lenina 51, Ekaterinburg 620000, Russia}

\email{sergey.gusb@gmail.com}

\begin{abstract}
An element $x$ of a lattice $L$ is modular if $L$ has no five-element sublattice isomorphic to the pentagon in which $x$ would correspond to the lonely midpoint.
In the present work, we classify all modular elements of the lattice of all monoid varieties.
\end{abstract}



\maketitle

\section{Introduction and summary}
\label{Sec: introduction}

An element $x$ of a lattice $L$ is \textit{modular} if it makes the formula
\[
\forall y,z\in L:\,y\le z\longrightarrow(x\vee y)\wedge z=(x\wedge z)\vee y
\]
hold true. 
Modular elements enjoy a distinguished position in lattices.
The interest in them is explained by the fact that modular elements of $L$ are exactly ones that are not the central elements of any pentagon (that is, a five-element non-modular sublattice shown in Fig.~\ref{pentagon}) of $L$ (see~\cite[Proposition~2.1]{Jezek-81}).  

\begin{figure}[htb]
\begin{center}
\unitlength=1mm
\linethickness{0.4pt}
\begin{picture}(20,31)
\put(10,0){\circle*{1.33}}
\put(0,10){\circle*{1.33}}
\put(0,20){\circle*{1.33}}
\put(20,15){\circle*{1.33}}
\put(10,30){\circle*{1.33}}

\gasset{AHnb=0,linewidth=0.4}
\drawline(10,0)(0,10)(0,20)(10,30)(20,15)(10,0)
\end{picture}
\caption{}
\label{pentagon}
\end{center}
\end{figure}

Modular elements played a crucial role in the study of first-order definability in the lattice $\Sem$ of all semigroup varieties~\cite{Jezek-McKenzie-93}. 
Although a complete description of modular elements in $\Sem$ is still unknown, a number of profound results have been obtained in this direction.
In particular, the set of all modular elements in $\Sem$ is uncountably infinite.
More information can be found in the comprehensive survey article~\cite{Vernikov-15} in the context of studying special elements of various types in the lattice $\Sem$.

The present article is concerned with modular elements of the lattice $\Mon$ of all varieties of monoids, i.e., semigroups with an identity element. 
Even though monoids are very similar to semigroups, the situation turns out to be very different.
Compared with lattice $\Sem$, the systematic study of lattice $\Mon$ has began relatively recently, although the first results in this direction were obtained back in the late 1960s.
In particular, the problem of describing modular elements of the lattice $\Mon$ remained open (see Section~9.1 in the recent survey~\cite{Gusev-Lee-Vernikov-22}).
The goal of the present note is a complete solution to this problem.
We present an exhaustive countably infinite list of monoid varieties that are modular elements of the lattice $\Mon$.

Let us briefly recall a few notions that we need to formulate our main result.
Let $\mathcal X$ be a countably infinite set called an \textit{alphabet}. 
As usual, let~$\mathcal X^\ast$ denote the free monoid over the alphabet~$\mathcal X$. 
Elements of~$\mathcal X$ are called \textit{letters} and elements of~$\mathcal X^\ast$ are called \textit{words}.
We treat the identity element of~$\mathcal X^\ast$ as \textit{the empty word}, which is denoted by~$1$.  
Words and letters are denoted by small Latin letters. 
However, words unlike letters are written in bold. 
An identity is written as $\mathbf u \approx \mathbf v$, where $\mathbf u,\mathbf v \in \mathcal X^\ast$; it is \textit{non-trivial} if $\mathbf u \ne \mathbf v$.
A variety $\mathbf V$ \textit{satisfies} an identity $\mathbf u \approx \mathbf v$, if for any monoid $M\in \mathbf V$ and any substitution $\varphi\colon \mathcal X \to M$, the equality $\varphi(\mathbf u)=\varphi(\mathbf v)$ holds in $M$.

For any set $\mathcal W$ of words, let $M(\mathcal W)$ denote the Rees quotient monoid of $\mathcal X^\ast$ over the ideal of all words that are not subwords of any word in $\mathcal W$.
Given a set $\mathcal W$ of words, let $\mathbf M(\mathcal W)$ denote the monoid variety generated by $M(\mathcal W)$.
For brevity, if $\mathbf w_1,\dots,\mathbf w_k\in\mathcal X^\ast$, then we write $M(\mathbf w_1,\dots,\mathbf w_k)$ [respectively, $\mathbf M(\mathbf w_1,\dots,\mathbf w_k)$] rather than $M(\{\mathbf w_1,\dots,\mathbf w_k\})$ [respectively, $\mathbf M(\{\mathbf w_1,\dots,\mathbf w_k\})$].

As usual, let $\mathbb N$ denote the set of all natural numbers. 
For any $n\in\mathbb N$, we denote by $S_n$ the full symmetric group on the set $\{1,\dots,n\}$.  
For any $n,m\in\mathbb N$ and $\rho\in S_{n+m}$, we define the words:
\begin{align*}
\mathbf a_{n,m}[\rho]&:=\biggl(\prod_{i=1}^n z_it_i\biggr)x\biggl(\prod_{i=1}^{n+m} z_{i\rho}\biggr)x\biggl(\prod_{i=n+1}^{n+m} t_iz_i\biggr),\\
\mathbf a_{n,m}^\prime[\rho]&:=\biggl(\prod_{i=1}^n z_it_i\biggr)\biggl(\prod_{i=1}^{n+m} z_{i\rho}\biggr)x^2\biggl(\prod_{i=n+1}^{n+m} t_iz_i\biggr).
\end{align*}
We denote by $\vMon$ the variety of all monoids.

Our main result is the following

\begin{theorem}
\label{T: modular}
For a monoid variety $\mathbf V$ the following are equivalent:
\begin{itemize}
\item[\textup{(i)}] $\mathbf V$ is a modular element of the lattice $\Mon$;
\item[\textup{(ii)}] $\mathbf V$ is either the variety $\vMon$ or satisfies the identities
\[
\begin{aligned}
&x^2\approx x^3,\,x^2y\approx yx^2,\,xyzxty\approx yxzxty,\,xzytxy\approx xzytyx,\\
&xzxtxysy\approx xzxtyxsy,\,xzxytysy\approx xzyxtysy,\,\mathbf a_{n,m}[\rho]\approx \mathbf a_{n,m}^\prime[\rho]
\end{aligned}
\]
for all $n,m\in\mathbb N$ and $\rho\in S_{n+m}$.
\item[\textup{(iii)}] $\mathbf V$ coincides with one of the varieties 
\begin{align*}
&\mathbf M(\emptyset),\,\mathbf M(1),\mathbf M(x),\,\mathbf M(xy),\,\mathbf M(xt_1x),\,\dots,\, \mathbf M(xt_1x\cdots t_nx),\,\dots,\,\mathbf M(\{xt_1x\cdots t_nx\mid n\in\mathbb N\}),\\
&\mathbf M(xzxyty,xt_1x),\,\dots,\, \mathbf M(xzxyty,xt_1x\cdots t_nx),\,\dots,\,\mathbf M(\{xzxyty,xt_1x\cdots t_nx\mid n\in\mathbb N\}),\,\vMon.
\end{align*}
\end{itemize}
\end{theorem}

In general, the set of modular elements in a lattice need not form a sublattice.
For example, the elements $x$ and $y$ of the lattice in Fig.~\ref{join is not modular} are modular but their join $x\vee y$ is not.
However, Theorem~\ref{T: modular} shows that the set of all modular elements of the lattice $\Mon$ forms a sublattice and, moreover, all proper monoid varieties that are modular elements in $\Mon$ constitute an order ideal in $\Mon$.

\begin{figure}[htb]
\begin{center}
\unitlength=1mm
\linethickness{0.4pt}
\begin{picture}(60,31)

\put(30,0){\circle*{1.33}}

\put(10,10){\circle*{1.33}}
\put(50,10){\circle*{1.33}}

\put(30,10){\circle*{1.33}}
\put(30,20){\circle*{1.33}}

\put(40,20){\circle*{1.33}}

\put(10,20){\circle*{1.33}}
\put(50,20){\circle*{1.33}}
\put(30,30){\circle*{1.33}}
\gasset{AHnb=0,linewidth=0.4}
\drawline(30,0)(10,10)(10,20)(30,10)(50,20)(50,10)(30,0)
\drawline(10,20)(30,30)(50,20)
\drawline(30,0)(30,10)
\drawline(30,10)(40,20)(30,30)(30,20)(10,10)
\drawline(30,20)(50,10)
\put(9,7.5){\makebox(0,0){$x$}}
\put(51,7.5){\makebox(0,0){$y$}}
\end{picture}
\caption{}
\label{join is not modular}
\end{center}
\end{figure}

The article consists of three sections. 
Some definitions, notation and auxiliary results are given in Section~\ref{Sec: preliminaries}, while Section~\ref{Sec: proof} is devoted to the proof of Theorem~\ref{T: modular}.

\section{Preliminaries}
\label{Sec: preliminaries}

An identity $\mathbf u \approx \mathbf v$ is \textit{directly deducible} from an identity $\mathbf s \approx \mathbf t$ if there exist some words $\mathbf a,\mathbf b \in \mathcal X^\ast$ and substitution $\varphi\colon \mathcal X \to \mathcal X^\ast$ such that $\{ \mathbf u, \mathbf v \} = \{ \mathbf a\varphi(\mathbf s)\mathbf b,\mathbf a\varphi(\mathbf t)\mathbf b \}$.
A non-trivial identity $\mathbf u \approx \mathbf v$ is \textit{deducible} from a set $\Sigma$ of identities if there exists some finite sequence $\mathbf u = \mathbf w_0, \dots, \mathbf w_m = \mathbf v$ of words such that each identity $\mathbf w_i \approx \mathbf w_{i+1}$ is directly deducible from some identity in $\Sigma$.

\begin{proposition}[Birkhoff's Completeness Theorem for Equational Logic; see {\cite[Theorem~II.14.19]{Burris-Sankappanavar-81}}]
\label{P: deduction}
A monoid variety  defined by a set $\Sigma$ of identities satisfies an identity $\mathbf u \approx \mathbf v$ if and only if $\mathbf u \approx \mathbf v$ is deducible from $\Sigma$.\qed
\end{proposition}

Given a variety $\mathbf V$, a word $\mathbf u$ is called an \emph{isoterm for} $\mathbf V$ if the only word $\mathbf v$ such that $\mathbf V$ satisfies the identity $\mathbf u \approx \mathbf v$ is the word $\mathbf u$ itself.

\begin{lemma}[\mdseries{\!\!\cite[Lemma~3.3]{Jackson-05}}]
\label{L: M(W) in V}
Let $\mathbf V$ be a monoid variety and $\mathcal W$ a set of words. 
Then $M(\mathcal W)\in\mathbf V$ if and only if each word in $\mathcal W$ is an isoterm for $\mathbf V$.\qed
\end{lemma}

The \textit{alphabet} of a word $\mathbf w$, i.e., the set of all letters occurring in $\mathbf w$, is denoted by $\alf(\mathbf w)$. 
For a word $\mathbf w$ and a letter $x$, let $\occ_x(\mathbf w)$ denote the number of occurrences of $x$ in $\mathbf w$.
A letter $x$ is called \textit{simple} [\textit{multiple}] \textit{in a word} $\mathbf w$ if $\occ_x(\mathbf w)=1$ [respectively, $\occ_x(\mathbf w)>1$]. 
The set of all simple [multiple] letters in a word $\mathbf w$ is denoted by $\simple(\mathbf w)$ [respectively, $\mul(\mathbf w)$]. 

\begin{lemma}[\mdseries{\!\!\cite[Lemma~2.17]{Gusev-Vernikov-21}}]
\label{L: id M(xy)}
Let $\mathbf u\approx \mathbf v$ be an identity of $M(xy)$. 
If $\mathbf u=\mathbf u_0t_1\mathbf u_1\cdots t_m\mathbf u_m$, where $\simple(\mathbf u)=\{t_1,\dots,t_m\}$, then $\mathbf v=\mathbf v_0t_1\mathbf v_1\cdots t_m\mathbf v_m$, $\alf(\mathbf u_0\cdots \mathbf u_m)=\alf(\mathbf v_0\cdots \mathbf v_m)$ and $\simple(\mathbf v)=\{t_1,\dots,t_m\}$.\qed
\end{lemma}

The following statement was established in the proof of Lemma~3.5 in~\cite{Gusev-Vernikov-21}.

\begin{lemma}
\label{L: swapping in linear-balanced}
Let $\mathbf V$ be a monoid variety such that $M(xt_1x\cdots t_nx)\in\mathbf V$. 
If $M(\mathbf p\,xy\,\mathbf q)\notin\mathbf V$, where $\mathbf p:=a_1t_1\cdots a_kt_k$ and $\mathbf q:=t_{k+1}a_{k+1}\cdots t_{k+\ell}a_{k+\ell}$ for some $k,\ell\ge0$ and $a_1,\dots,a_{k+\ell}$ are letters such that $\{a_1,\dots,a_{k+\ell}\}=\{x,y\}$ and $\occ_x(\mathbf p\mathbf q),\occ_y(\mathbf p\mathbf q)\le n$, then $\mathbf V$ satisfies the identity $\mathbf p\,xy\,\mathbf q\approx\mathbf p\,yx\,\mathbf q$.\qed
\end{lemma}

If $\mathbf w$ is a word and $\mathcal Z\subseteq\alf(\mathbf w)$, then we denote by $\mathbf w_{\mathcal Z}$ [respectively, $\mathbf w(\mathcal Z)$] the word obtained from $\mathbf w$ by removing all occurrences of letters from $\mathcal Z$ [respectively, $\alf(\mathbf w)\setminus \mathcal Z$]. 
If $\mathcal Z=\{z\}$, then we write $\mathbf w_z$ rather than $\mathbf w_{\{z\}}$.

The expression $_{i\mathbf w}x$ means the $i$th occurrence of a letter $x$ in a word $\mathbf w$. 
If the $i$th occurrence of $x$ precedes the $j$th occurrence of $y$ in a word $\mathbf w$, then we write $({_{i\mathbf w}x}) < ({_{j\mathbf w}y})$.

%

\section{Proof of Theorem~\ref{T: modular}}
\label{Sec: proof}

We will prove Theorem~\ref{T: modular} following the scheme (i) $\Rightarrow$ (ii) $\Rightarrow$ (iii) $\Rightarrow$ (i).

\smallskip

(i) $\Rightarrow$ (ii). 
Let $\mathbf V$ be a proper monoid variety which is a modular element of the lattice $\Mon$.
Then the variety $\mathbf V$ satisfies the identities
\begin{equation}
\label{xx=xxx,xxy=yxx}
x^2\approx x^3,\  x^2y\approx yx^2
\end{equation} 
by Proposition~4.3 in~\cite{Gusev-Lee-21}.

If $M(xyx)\notin \mathbf V$, then it follows from Lemma~3.3(i) in~\cite{Gusev-Vernikov-21} that $\mathbf V$ satisfies the identities $x^2y\approx xyx\approx yx^2$ which, evidently, imply all the identities listed in Item~(ii) of Theorem~\ref{T: modular}.
So, we may further assume that $M(xyx)\in \mathbf V$.

The rest of the proof proceeds in three steps.

\smallskip

\noindent\textbf{Step 1:} $\mathbf V$ satisfies the identities $xyzxty\approx yxzxty$ and $xzytxy\approx xzytyx$.

Arguing by contradiction, suppose that $\mathbf V$ violates the identity $xyzxty\approx yxzxty$.  
Then $M(xyzxty)\in\mathbf V$ by Lemma~\ref{L: swapping in linear-balanced}.
Let $\mathbf X^\prime$ denote the monoid variety defined by the identity 
\[
\mathbf u:=z_1t_1z_2t_2\,c^2z_1bz_2\,xcy\,b\,s_1xs_2y\approx z_1t_1z_2t_2\,c^2z_1bz_2\,ycx\,b\,s_1xs_2y=:\mathbf v.
\]
We need the following auxiliary result.

\begin{claim}
\label{C: X'}
The words $\mathbf u$ and $\mathbf v$ can only form an identity of $\mathbf X^\prime$ with each other.
\end{claim}

\begin{proof}
Take $\mathbf w\in\{\mathbf u,\mathbf v\}$ and consider an arbitrary identity of the form $\mathbf w\approx \mathbf w^\prime$ that holds in $\mathbf X^\prime$.
We are going to verify that $\mathbf w^\prime \in\{\mathbf u,\mathbf v\}$.
By Proposition~\ref{P: deduction} and evident induction, we may assume without any loss that the identity $\mathbf w\approx \mathbf w^\prime$ is directly deducible from the identity $\mathbf u\approx \mathbf v$, i.e., there exist words $\mathbf a,\mathbf b \in \mathcal X^\ast$ and a substitution $\varphi\colon \mathcal X \to \mathcal X^\ast$ such that $(\mathbf w, \mathbf w^\prime ) = ( \mathbf a\varphi(\mathbf u)\mathbf b,\mathbf a\varphi(\mathbf v)\mathbf b )$.
Notice that every subword of $\mathbf w$ of length $>1$ occurs in $\mathbf w$ exactly once and each letter occurs in $\mathbf w$ at most thrice.
It follows that
\begin{itemize}
\item[\textup{($\ast$)}] $\varphi(v)$ is either the empty word or a letter for any $v\in\mul(\mathbf u)=\mul(\mathbf v)$.
\end{itemize}

Suppose first that $\varphi(c)=1$.
In this case, if $\varphi(x)=1$ or $\varphi(y)=1$, then $\varphi(\mathbf u)=\varphi(\mathbf v)$ and so $\mathbf w^\prime\in\{\mathbf u,\mathbf v\}$, as required.
So, we may assume that $\varphi(x)\ne1$ and $\varphi(y)\ne1$.
Then $\varphi(x)$ and $\varphi(y)$ are letters by~($\ast$). 
These letters must be distinct since no letter occurs in $\mathbf w$ more than thrice.
However, the word $\mathbf w$ does not contain any subword consisting exclusively of non-last occurrences of two distinct letters in $\mathbf w$, while the first occurrences of the letters $\varphi(x)$ and $\varphi(y)$ are adjacent in $\mathbf w$, a contradiction.

Suppose now that $\varphi(c)\ne1$.
Then $\varphi(c)=c$ because $c$ is the only letter which occurs thrice in both $\mathbf u$ and $\mathbf w$.
Hence
\[
\varphi(z_1bz_2x)=
\begin{cases}
z_1bz_2x&\text{if } \mathbf w=\mathbf u,\\
z_1bz_2y&\text{if } \mathbf w=\mathbf v.
\end{cases}
\]
Further, it follows from~($\ast)$ that $\varphi(b)=b$, whence 
\[
(\varphi(x),\varphi(y))=
\begin{cases}
(x,y)&\text{if } \mathbf w=\mathbf u,\\
(y,x)&\text{if } \mathbf w=\mathbf v.
\end{cases}
\]
Since $({_{1\mathbf u}}x)<({_{1\mathbf u}}y)<({_{2\mathbf u}}x)<({_{2\mathbf u}}y)$ and $({_{1\mathbf v}}y)<({_{1\mathbf v}}x)<({_{2\mathbf v}}x)<({_{2\mathbf v}}y)$, it follows that the case when $\mathbf w=\mathbf v$ is impossible.
Therefore, $(\varphi(x),\varphi(y))=(x,y)$ in either case. 
This implies that $\mathbf w^\prime =\mathbf v$, and we are done.
\end{proof}

Put $\mathbf X:=\mathbf M(\mathbf u,\mathbf v)\wedge\mathbf X^\prime$.
Consider an arbitrary identity of the form $\mathbf u\approx \mathbf w$ that is satisfied by the variety $\mathbf X\vee \mathbf V$.
In view of Claim~\ref{C: X'}, $\mathbf w\in \{\mathbf u,\mathbf v\}$.
Since $M(xyzxty)\in\mathbf V$, Lemma~\ref{L: M(W) in V} implies that $\mathbf w(x,y,s_1,s_2)=\mathbf u(x,y,s_1,s_2)=xys_1xs_2y$.
Hence $\mathbf w = \mathbf u$.
We see that $\mathbf u$ is an isoterm for the variety $\mathbf X\vee \mathbf V$.
By a similar argument, we can show that $\mathbf v$ is an isoterm for $\mathbf X\vee \mathbf V$ as well.
Then 
\[
\left(\mathbf X\vee \mathbf V\right)\wedge \mathbf M(\mathbf u,\mathbf v) = \mathbf M(\mathbf u,\mathbf v)
\]
by Lemma~\ref{L: M(W) in V}.
It is easy to see that the identity $xyzxty\approx yxzxty$ holds in $M(\mathbf u,\mathbf v)$.
Then $\mathbf V\wedge\mathbf M(\mathbf u,\mathbf v)$ satisfies $\mathbf u\stackrel{\mathbf V}\approx \mathbf u_cc^2\stackrel{\mathbf M(\mathbf u,\mathbf v)}\approx \mathbf v_cc^2\stackrel{\mathbf V}\approx \mathbf v$.
By the very definition, the identity $\mathbf u\approx\mathbf v$ is satisfied by $\mathbf X$ as well.
Since $\mathbf X\subset\mathbf M(\mathbf u,\mathbf v)$, we have
\[
\mathbf X\vee\left(\mathbf V\wedge\mathbf M(\mathbf u,\mathbf v)\right)\subset \left(\mathbf X\vee \mathbf V\right)\wedge \mathbf M(\mathbf u,\mathbf v) = \mathbf M(\mathbf u,\mathbf v),
\]
contradicting the fact that $\mathbf V$ is a modular element of $\Mon$.
Therefore, $\mathbf V$ satisfies $xyzxty\approx yxzxty$.
By the dual argument, we can show that $xzytxy\approx xzytyx$ holds in $\mathbf V$.

\smallskip

\noindent\textbf{Step 2:} $\mathbf V$ satisfies the identity $\mathbf a_{n,m}[\rho]\approx \mathbf a_{n,m}^\prime[\rho]$ for all $n,m\in \mathbb N$ and $\rho\in S_{n+m}$.

Arguing by contradiction, suppose that $\mathbf V$ violates $\mathbf a_{p,q}[\pi]\approx \mathbf a_{p,q}^\prime[\pi]$ for some $p,q\in \mathbb N$ and $\pi\in S_{p+q}$.
Then $M(\mathbf a_{n,m}[\rho])\in\mathbf V$ for some $n,m\in \mathbb N$ and $\rho\in S_{n+m}$ by Lemma 4.8 in~\cite{Gusev-23}.
For brevity, put
\begin{align*}
&\mathbf p:=
\begin{cases}
z_1t_1\cdots z_nt_n,&\text{if }1\le 1\rho\le n,\\
y_1s_1y_2s_2\,z_1t_1\cdots z_nt_n,&\text{if }n< 1\rho\le n+m,
\end{cases}
\\
&\mathbf r:=
\begin{cases}
z_{n+1}t_{n+1}\cdots z_{n+m}t_{n+m}\,s_1y_1s_2y_2,&\text{if }1\le 1\rho\le n,\\
z_{n+1}t_{n+1}\cdots z_{n+m}t_{n+m},&\text{if }n< 1\rho\le n+m.
\end{cases}
\end{align*}
Let
\[
\mathbf X:=
\begin{cases}
\mathbf M(y^2xtx,xy^2tx)&\text{if }1\le 1\rho\le n,\\
\mathbf M(xtxy^2,xty^2x)&\text{if }n< 1\rho\le n+m.
\end{cases}
\]
Consider an arbitrary identity of the form 
\[
\mathbf u:=\mathbf px\biggl(\prod_{i=1}^{n+m-1} z_{i\rho}\biggr)y_1z^2y_2z_{(n+m)\rho}x\mathbf r\approx \mathbf v
\] 
that is satisfied by the variety $\mathbf X\vee \mathbf V$.
In view of Lemma~\ref{L: M(W) in V}, 
\begin{align*}
&\mathbf v_{\{y_1,y_2,s_1,s_2,z\}}=\mathbf u_{\{y_1,y_2,s_1,s_2,z\}}=\mathbf a_{n,m}[\rho],\\
&\mathbf v_{\{x,z\}}=\mathbf u_{\{x,z\}}=\mathbf p\biggl(\prod_{i=1}^{n+m-1} z_{i\rho}\biggr)y_1y_2z_{(n+m)\rho}\mathbf r,\\
&\mathbf v(y_1,s_1,z)=\mathbf u(y_1,s_1,z)=
\begin{cases}
y_1z^2s_1y_1&\text{if }1\le 1\rho\le n,\\
y_1s_1y_1z^2&\text{if }n< 1\rho\le n+m,
\end{cases}
\\
&\mathbf v(y_2,s_2,z)=\mathbf u(y_2,s_2,z)=
\begin{cases}
z^2y_2s_2y_2&\text{if }1\le 1\rho\le n,\\
y_2s_2z^2y_2&\text{if }n< 1\rho\le n+m.
\end{cases}
\end{align*}
It follows that $\mathbf v = \mathbf u$.
We see that the word $\mathbf u$ is an isoterm for the variety $\mathbf X\vee \mathbf V$.
Then 
\[
\left(\mathbf X\vee \mathbf V\right)\wedge \mathbf M(\mathbf u) = \mathbf M(\mathbf u)
\]
by Lemma~\ref{L: M(W) in V}.

Now we are going to verify that $M(\mathbf u)$ satisfies the identity
\begin{align*}
\mathbf a:=\mathbf u_z=\mathbf px\biggl(\prod_{i=1}^{n+m-1} z_{i\rho}\biggr)y_1y_2z_{(n+m)\rho}x\mathbf r
\approx\mathbf pz_{1\rho}x\biggl(\prod_{i=2}^{n+m-1} z_{i\rho}\biggr)y_1y_2z_{(n+m)\rho}x\mathbf r=:\mathbf a^\prime.
\end{align*}
Consider an arbitrary substitution $\psi\colon \mathscr X \to M(\mathbf u)$ and show that $\psi(\mathbf a)=\psi(\mathbf a^\prime)$.
We may suppose that at least one of the elements $\psi(\mathbf a)$ or $\psi(\mathbf a^\prime)$ is non-zero and forms a subword of $\mathbf u$. 
If $\psi(x)=1$, then $\psi(\mathbf a)=\psi(\mathbf a^\prime)$, and we are done.
So, we may further assume that $\psi(x)\ne1$.
Clearly, none of the letters $s_1,s_2,t_1,\dots,t_{n+m}$ belongs to $\alf(\psi(x))$ because all these letters are simple in $\mathbf u$, while $x\in\mul(\mathbf a)=\mul(\mathbf a^\prime)$.
Further, none of the letters $y_1,y_2,z_1,\dots,z_{n+m}$ belongs to $\alf(\psi(x))$ because there is a simple letter between the first and the second occurrences of any of these letters in $\mathbf u$, while there are no simple letters between the first and the second occurrences of $x$ in both $\mathbf a$ and $\mathbf a^\prime$.
Finally, $x\notin \alf(\psi(x))$ because 
\[
\psi\left(\left(\prod_{i=1}^{n+m-1} z_{i\rho}\biggr)y_1y_2z_{(n+m)\rho}\right)\right)
\] 
cannot contain $z^2$ as a subword.
Therefore, $\psi(x)=z$.
Then $\psi(y_1)=\psi(y_2)=\psi(z_{2\rho})=\cdots=\psi(z_{(n+m)\rho})=1$.
Assume that $\psi(z_{1\rho})\ne1$.
Then $\psi(\mathbf a^\prime)$ is a subword of $\mathbf u$, and if $1\le 1\rho\le n$ [respectively, $n< 1\rho\le n+m$] the image of the second [respectively, first] occurrence of $z_{1\rho}$ in $\mathbf a^\prime$ under $\psi$ must contain the first [respectively, second] occurrence of $y_1$ in $\mathbf u$, a contradiction.
Therefore, $\psi(z_{1\rho})=1$, whence $\psi(\mathbf a)=\psi(\mathbf a^\prime)$.
Thus, we have proved that the identity $\mathbf a\approx \mathbf a^\prime$ holds in $M(\mathbf u)$.

Then $\mathbf V\wedge\mathbf M(\mathbf u)$ satisfies
\[
\mathbf u\stackrel{\mathbf V}\approx \mathbf u_zz^2=\mathbf az^2\stackrel{\mathbf M(\mathbf u)}\approx  \mathbf a^\prime z^2\stackrel{\mathbf V}\approx \mathbf pz_{1\rho}x\biggl(\prod_{i=2}^{n+m-1} z_{i\rho}\biggr)y_1z^2y_2z_{(n+m)\rho}x\mathbf r=:\mathbf u^\prime.
\]
It follows from the very definition of $\mathbf X$ that $\mathbf X\subset\mathbf M(\mathbf u)$ and the identity $\mathbf u\approx \mathbf u^\prime$ is satisfied by $\mathbf X$.
Then
\[
\mathbf X\vee\left(\mathbf V\wedge\mathbf M(\mathbf u)\right)\subset \left(\mathbf X\vee \mathbf V\right)\wedge \mathbf M(\mathbf u) = \mathbf M(\mathbf u),
\]
contradicting the fact that $\mathbf V$ is a modular element of $\Mon$.
Therefore, $\mathbf V$ satisfies $\mathbf a_{n,m}[\rho]\approx \mathbf a_{n,m}^\prime[\rho]$ for all $n,m\in \mathbb N$ and $\rho\in S_{n+m}$.

\smallskip

\noindent\textbf{Step 3:} $\mathbf V$ satisfies the identities $xzxtxysy\approx xzxtyxsy$ and $xzxytysy\approx xzyxtysy$.

If $M(xyxzx)\notin \mathbf V$, then, by Lemma~\ref{L: M(W) in V}, the variety $\mathbf V$ satisfies an identity $xyxzx\approx x^pyx^qzx^r$ with either $p>1$ or $q>1$ or $r>1$.
In any case, this identity together with the identities~\eqref{xx=xxx,xxy=yxx} imply the identity $xyxzx\approx x^2yz$ and so the identities $xzxtxysy\approx xzxtyxsy$ and $xzxytysy\approx xzyxtysy$.
So, we may further assume that $M(xyxzx)\in \mathbf V$.

Arguing by contradiction, suppose that $\mathbf V$ violates the identity $xzxtxysy\approx xzxtyxsy$.
Then $M(xzxtxysy)\in\mathbf V$ by Lemma~\ref{L: swapping in linear-balanced}.
Let $\mathbf X^\prime$ denote the monoid variety defined by the identity 
\[
\mathbf u:=xs_1xs_2z_1t_1z_2t_2\,cz_1bz_2\,xcy\,b\,s_3y\approx xs_1xs_2z_1t_1z_2t_2\,cz_1bz_2\,ycx\,b\,s_3y=:\mathbf v.
\]
We need the following auxiliary result.

\begin{claim}
\label{C: X' 2}
The words $\mathbf u$ and $\mathbf v$ can only form an identity of $\mathbf X^\prime$ with each other.
\end{claim}

\begin{proof}
Take $\mathbf w\in\{\mathbf u,\mathbf v\}$ and consider an arbitrary identity of the form $\mathbf w\approx \mathbf w^\prime$ that holds in $\mathbf X^\prime$.
We are going to verify that $\mathbf w^\prime \in\{\mathbf u,\mathbf v\}$.
By Proposition~\ref{P: deduction} and evident induction, we may assume without any loss that the identity $\mathbf w\approx \mathbf w^\prime$ is directly deducible from the identity $\mathbf u\approx \mathbf v$, i.e., there exist some words $\mathbf a,\mathbf b \in \mathcal X^\ast$ and substitution $\varphi\colon \mathcal X \to \mathcal X^\ast$ such that $(\mathbf w, \mathbf w^\prime ) = ( \mathbf a\varphi(\mathbf u)\mathbf b,\mathbf a\varphi(\mathbf v)\mathbf b )$.
Notice that every subword of $\mathbf w$ of length $>1$ occurs in $\mathbf w$ exactly once.
It follows that
\begin{itemize}
\item[\textup{($\ast$)}] $\varphi(v)$ is either the empty word or a letter for any $v\in\mul(\mathbf u)$.
\end{itemize}

Suppose first that $\varphi(c)=1$.
In this case, if $\varphi(x)=1$ or $\varphi(y)=1$, then $\varphi(\mathbf u)=\varphi(\mathbf v)$ and so $\mathbf w^\prime\in\{\mathbf u,\mathbf v\}$, as required.
So, we may assume that $\varphi(x)\ne1$ and $\varphi(y)\ne1$.
Then $\varphi(x)$ and $\varphi(y)$ are letters by~($\ast$). 
Since $x$ is the only letter that occurs thrice in both $\mathbf u$ and $\mathbf v$, we have $\varphi(x)=x$.
It follows that the image of ${_{1\mathbf w}}y$ under $\varphi$ must be either ${_{2\mathbf w}}c$ (if $\mathbf w=\mathbf u$) or ${_{2\mathbf w}}b$ (if $\mathbf w=\mathbf v$), a contradiction.

Suppose now that $\varphi(c)\ne1$.
Then $\varphi(c)\in\{b,c\}$ because $b$ and $c$ are the only multiple letters of $\mathbf w$ between the first and second occurrences of which there are no simple letters.
If $\varphi(c)=b$, then 
\[
\varphi(z_1bz_2x)=
\begin{cases}
z_2xcy&\text{if } \mathbf w=\mathbf u,\\
z_2ycx&\text{if } \mathbf w=\mathbf v.
\end{cases}
\] 
and, by~($\ast$), $\varphi(b)\in\{x,y\}$.
However, this contradicts the fact that there are simple letters between ${_{2\mathbf w}}x$ and ${_{3\mathbf w}}x$ as well as between ${_{1\mathbf w}}y$ and ${_{2\mathbf w}}y$ in $\mathbf w$.
Hence $\varphi(c)=c$ and so
\[
\varphi(z_1bz_2x)=
\begin{cases}
z_1bz_2x&\text{if } \mathbf w=\mathbf u,\\
z_1bz_2y&\text{if } \mathbf w=\mathbf v.
\end{cases}
\]
Further, it follows from~($\ast)$ that $\varphi(b)=b$, whence 
\[
(\varphi(x),\varphi(y))=
\begin{cases}
(x,y)&\text{if } \mathbf w=\mathbf u,\\
(y,x)&\text{if } \mathbf w=\mathbf v.
\end{cases}
\]
Since $({_{1\mathbf u}}x)<({_{2\mathbf u}}x)<({_{3\mathbf u}}x)<({_{1\mathbf u}}y)<({_{2\mathbf u}}y)$ and $({_{1\mathbf v}}x)<({_{2\mathbf v}}x)<({_{1\mathbf v}}y)<({_{3\mathbf v}}x)<({_{2\mathbf v}}y)$, it follows that the case when $\mathbf w=\mathbf v$ is impossible.
Therefore, $(\varphi(x),\varphi(y))=(x,y)$ in either case. 
This implies that $\mathbf w^\prime =\mathbf v$, and we are done.
\end{proof}

Put $\mathbf X:=\mathbf M(\mathbf u,\mathbf v)\wedge\mathbf X^\prime$.
Consider an arbitrary identity of the form $\mathbf u\approx \mathbf w$ that is satisfied by the variety $\mathbf X\vee \mathbf V$.
In view of Claim~\ref{C: X' 2}, $\mathbf w\in \{\mathbf u,\mathbf v\}$.
Since $M(xzxtxysy)\in\mathbf V$, Lemma~\ref{L: M(W) in V} implies that $\mathbf w(x,y,s_1,s_2,s_3)=\mathbf u(x,y,s_1,s_2,s_3)=xs_1xs_2xys_3y$.
Hence $\mathbf w = \mathbf u$.
We see that $\mathbf u$ is an isoterm for the variety $\mathbf X\vee \mathbf V$.
By a similar argument, we can show that $\mathbf v$ is an isoterm for $\mathbf X\vee \mathbf V$ as well.
Then 
\[
\left(\mathbf X\vee \mathbf V\right)\wedge \mathbf M(\mathbf u,\mathbf v) = \mathbf M(\mathbf u,\mathbf v)
\]
by Lemma~\ref{L: M(W) in V}.
It is easy to see that $xzxtxysy\approx xzxtyxsy$ holds in $M(\mathbf u,\mathbf v)$.
Recall that $\mathbf V$ satisfies the identities~\eqref{xx=xxx,xxy=yxx} and $\mathbf a_{n,m}[\rho]\approx  \mathbf a_{n,m}^\prime[\rho]$ for all $n,m\in \mathbb N$ and $\rho\in S_{n+m}$.
Then $\mathbf V\wedge\mathbf M(\mathbf u,\mathbf v)$ satisfies $\mathbf u\stackrel{\mathbf V}\approx \mathbf u_cc^2\stackrel{\mathbf M(\mathbf u,\mathbf v)}\approx \mathbf v_cc^2\stackrel{\mathbf V}\approx \mathbf v$.
By the very definition, the identity $\mathbf u\approx\mathbf v$ is satisfied by $\mathbf X$ as well.
Since $\mathbf X\subset\mathbf M(\mathbf u,\mathbf v)$, we have
\[
\mathbf X\vee\left(\mathbf V\wedge\mathbf M(\mathbf u,\mathbf v)\right)\subset \left(\mathbf X\vee \mathbf V\right)\wedge \mathbf M(\mathbf u,\mathbf v) = \mathbf M(\mathbf u,\mathbf v),
\]
contradicting the fact that $\mathbf V$ is a modular element of $\Mon$.
Therefore, $\mathbf V$ satisfies $xzxtxysy\approx xzxtyxsy$.
By the dual argument, we can show that $xzxytysy\approx xzyxtysy$ holds in $\mathbf V$.

\smallskip

Implication (i) $\Rightarrow$ (ii) is thus proved.

\smallskip

(ii) $\Rightarrow$ (iii). 
If $M(xyx)\notin\mathbf V$, then it follows from Lemma~3.3(i) in~\cite{Gusev-Vernikov-21} that $\mathbf V$ coincides with one of the varieties $\mathbf M(\emptyset)$, $\mathbf M(1)$, $\mathbf M(x)$ or $\mathbf M(xy)$, and we are done.
Assume now that $M(xyx)\in\mathbf V$.
Denote by $\mathscr W$ the set of all words in $\{xzxyty,\,xt_1x\cdots t_nx\mid n\in\mathbb N\}$ that are isoterms for $\mathbf V$. 
Notice that, in this case, the variety $\mathbf M(\mathscr W)$ coincides with one of the varieties 
\begin{align*}
&\mathbf M(xt_1x),\,\dots,\, \mathbf M(xt_1x\cdots t_nx),\,\dots,\,\mathbf M(\{xt_1x\cdots t_nx\mid n\in\mathbb N\}),\\
&\mathbf M(xzxyty,xt_1x),\,\dots,\, \mathbf M(xzxyty,xt_1x\cdots t_nx),\,\dots,\,\mathbf M(\{xzxyty,xt_1x\cdots t_nx\mid n\in\mathbb N\}).
\end{align*}
According to Lemma~\ref{L: M(W) in V}, $\mathbf M(\mathscr W)\subseteq \mathbf V$. 
It remains to verify that $\mathbf V\subseteq \mathbf M(\mathscr W)$. 
Arguing by contradiction, suppose that $\mathbf M(\mathscr W)\ne \mathbf V$. 
Then there is an identity $\sigma$ which holds in $\mathbf M(\mathscr W)$ but does not hold in $\mathbf V$. 
Proposition 3.15 in~\cite{Gusev-Vernikov-21} and its proof allow us to assume that $\sigma$ coincides with either the identity
\begin{equation}
\label{one letter in a block}
x\biggl(\prod_{i=1}^nt_ix\biggr)\approx x^2\biggl(\prod_{i=1}^nt_i\biggr),
\end{equation} 
where $n\in\mathbb N$, or the identity
\begin{equation}
\label{two letters in a block}
\biggl(\prod_{i=1}^k a_it_i\biggr) xy \biggl(\prod_{i=k+1}^{k+\ell} t_i a_i\biggr)\approx\biggl(\prod_{i=1}^k a_it_i\biggr) yx  \biggl(\prod_{i=k+1}^{k+\ell} t_ia_i\biggr),
\end{equation} 
where $k,\ell\ge0$ and $\{a_1,\dots,a_{k+\ell}\}=\{x,y\}$.
Notice that an identity of the form~\eqref{two letters in a block} does not hold in $\mathbf V$ if and only if it coincides (up to renaming of letters) with the identity $xzxyty\approx xzyxty$.
If $\sigma$ equals to~\eqref{one letter in a block}, the word $xt_1x\cdots t_nx$ is an isoterm for $\mathbf V$~by Lemma 3.3(ii) in~\cite{Gusev-Vernikov-21}.
Then this word must belong to the set $\mathscr W$, contradicting our assumption that $\sigma$ is satisfied by $\mathbf M(\mathscr W)$.
If $\sigma$ coincides with the identity $xzxyty\approx xzyxty$, then the word $xzxyty$ is an isoterm for $\mathbf V$ by Lemma~\ref{L: swapping in linear-balanced}.
But this contradicts our assumption that $\sigma$ holds in $\mathbf M(\mathscr W)$.
Therefore, $\mathbf V= \mathbf M(\mathscr W)$, and we are done.

\smallskip

(iii) $\Rightarrow$ (i). 
It follows from~\cite[Theorem~1.1]{Gusev-Lee-21} that the varieties $\mathbf M(\emptyset)$, $\mathbf M(1)$, $\mathbf M(x)$ and $\mathbf M(xy)$ are modular elements of $\Mon$.
So, we may further assume that $\mathbf V=\mathbf M(\mathscr W)$ for some $\mathscr W\subseteq\{xzxyty,xt_1x\cdots t_nx\mid n\in\mathbb N\}$.
In particular, $M(xyx)\in\mathbf V$.

Arguing by contradiction, suppose that $\mathbf V$ is not a modular element of the lattice $\Mon$. 
This means that there are monoid varieties $\mathbf X$ and $\mathbf Y$ such that $\mathbf X\vee \mathbf V=\mathbf Y\vee \mathbf V$, $\mathbf X\wedge \mathbf V=\mathbf Y\wedge \mathbf V$ but $\mathbf X\subset\mathbf Y$.

We need the following auxiliary result.

\begin{claim}
\label{C: X=Y}
Let $\mathbf u\approx \mathbf v$ be an identity of $\mathbf X$. 
If $\mathbf u\approx \mathbf v$ holds in $\mathbf M(\{xt_1x\cdots t_nx\mid n\in\mathbb N\})$, then $\mathbf u\approx \mathbf v$ is satisfied by $\mathbf Y$ as well.
\end{claim}

\begin{proof}
If the identity $\mathbf u\approx \mathbf v$ holds in $\mathbf V$, then it is also holds in $\mathbf X\vee\mathbf V$ and so in $\mathbf Y$ because $\mathbf X\vee \mathbf V=\mathbf Y\vee \mathbf V$.
So, we may further assume that $\mathbf V$ violates $\mathbf u\approx \mathbf v$.
This is only possible when $M(xzxyty)$ violates $\mathbf u\approx \mathbf v$ and $M(xzxyty)\in\mathbf V$ because $\mathbf V\subseteq\mathbf M(\{xzxyty,xt_1x\cdots t_nx\mid n\in\mathbb N\})$ and $\mathbf u\approx \mathbf v$ is satisfied by $M(\{xt_1x\cdots t_nx\mid n\in\mathbb N\})$.
Then $M(xzxyty)\notin\mathbf Y$ because $M(xzxyty)\in\mathbf X$ otherwise, contradicting the fact that $M(xzxyty)$ violates $\mathbf u\approx \mathbf v$.

If $M(xyx)\in\mathbf Y$, then $\mathbf Y$ satisfies $xzxyty\approx xzyxty$ by Lemma~\ref{L: swapping in linear-balanced}.
Now let $M(xyx)\notin\mathbf Y$.
Then Lemma~\ref{L: M(W) in V} implies that the variety $\mathbf Y$ satisfies an identity $xyx\approx x^pyx^q$ with either $p>1$ or $q>1$.
Further, it follows from Lemma~\ref{L: swapping in linear-balanced} that $\mathbf X\vee\mathbf M(xyx)$ satisfies $xzxyty\approx xzyxty$.
Then $\mathbf X$ satisfies also the identity $x^pzx^qyty\approx x^pzyx^qty$.
Evidently, the last identity holds in the variety $\mathbf V$ as well.
Therefore, it is satisfied by $\mathbf Y$ because $\mathbf Y\subset\mathbf X\vee\mathbf V$.
Thus, we see that the identity $xzxyty\approx xzyxty$ holds in the variety $\mathbf Y$ in any case.


Let $\mathbf u=\mathbf u_0t_1\mathbf u_1\cdots t_m\mathbf u_m$, where $t_1,\dots,t_m$ are all the simple letters of the word $\mathbf u$ and $\mathbf u_0,\dots,\mathbf u_m\in\mathscr X^\ast$.
Then, by Lemma~\ref{L: id M(xy)}, $\simple(\mathbf v)=\{t_1,\dots,t_m\}$ and $\mathbf v=\mathbf v_0t_1\mathbf v_1\cdots t_m\mathbf v_m$ for some $\mathbf v_0,\dots,\mathbf v_m\in\mathscr X^\ast$.

It is easy to see that the identity $xzxyty\approx xzyxty$ implies the identities
\[
\mathbf u\approx \mathbf p_0\mathbf q_0t_1 \mathbf p_1\mathbf q_1\cdots t_m \mathbf p_m\mathbf q_m=:\mathbf u^\prime,\ \ 
\mathbf v\approx \mathbf p_0^\prime\mathbf q_0^\prime t_1 \mathbf p_1^\prime\mathbf q_1^\prime\cdots t_m \mathbf p_m^\prime\mathbf q_m^\prime=:\mathbf v^\prime
\]
where the word $\mathbf p_i$ [respectively, $\mathbf q_i$] is obtained from the word $\mathbf u_i$ by retaining only the first [respectively, non-first] occurrences of letters in $\mathbf u$, while the word $\mathbf p_i^\prime$ [respectively, $\mathbf q_i^\prime$] is obtained from the word $\mathbf v_i$ by retaining only the first [respectively, non-first] occurrences of letters in $\mathbf v$.
By the very construction of the word $\mathbf u^\prime$ and $\mathbf v^\prime$, the identity $\mathbf u^\prime\approx \mathbf v^\prime$ holds in the monoid $M(xzxyty)$ and so in the variety $\mathbf V$.
Since $\mathbf X$ satisfies $xzxyty\approx xzyxty$ and $\mathbf u\approx \mathbf v$, the identity $\mathbf u^\prime\approx \mathbf v^\prime$ also holds in $\mathbf X$.  
Then $\mathbf Y$ satisfies $\mathbf u^\prime\approx \mathbf v^\prime$ as $\mathbf Y\subseteq\mathbf X\vee\mathbf V$.
It follows that $\mathbf u\approx \mathbf v$ holds in $\mathbf Y$ because the identities $\mathbf u\approx \mathbf u^\prime$ and $\mathbf v\approx \mathbf v^\prime$ are consequences of $xzxyty\approx xzyxty$.
\end{proof}

It follows from Lemma~\ref{L: M(W) in V} and Claim~\ref{C: X=Y} that there is the least $k\in\mathbb N$ such that the word $xt_1x\cdots t_kx$ is not an isoterm for $\mathbf X$. 
Then $\mathbf X$ satisfies an non-trivial identity $xt_1x\cdots t_kx\approx \mathbf w$ for some $\mathbf w\in\mathscr X^\ast$.
It follows from Lemma~\ref{L: id M(xy)} that $\mathbf w=x^{e_0}t_1x^{e_1}\cdots t_kx^{e_k}$ for some $e_0,\dots,e_k\ge0$. 
Since $k$ is the least natural number for which the word $xt_1x\cdots t_kx$ is not an isoterm for $\mathbf X$, at least one of the numbers $e_0,\dots,e_k$ must exceed $1$.

Now consider an arbitrary identity $\mathbf a\approx \mathbf b$ of $\mathbf X$.
Applying the identity $xt_1x\cdots t_kx\approx \mathbf w$ to the words $\mathbf a$ and $\mathbf b$, we can replace their subwords of the form $x\mathbf f_1x\cdots \mathbf f_kx$ to $x^{e_0}\mathbf f_1x^{e_1}\cdots \mathbf f_kx^{e_k}$.
In other words, the identity $xt_1x\cdots t_kx\approx \mathbf w$ can be used to convert the words $\mathbf a$ and $\mathbf b$ into some words $\mathbf a^\prime$ and $\mathbf b^\prime$, respectively, so that the identity $\mathbf a^\prime\approx\mathbf b^\prime$ holds in $M(\{xt_1x\cdots t_nx\mid n\in\mathbb N\})$ (because the word $xt_1x\cdots t_{k-1}x$ is an isoterm for $\mathbf X$ and at least one of the numbers $e_0,\dots,e_k$ exceeds $1$).
Therefore, the variety $\mathbf X$ can be defined by $\{xt_1x\cdots t_kx\approx \mathbf w\}\cup\Sigma$ for some set $\Sigma$ of identities holding in $M(\{xt_1x\cdots t_nx\mid n\in\mathbb N\})$.

Further, if $M(xt_1x\cdots t_kx)\in\mathbf V$, then $M(xt_1x\cdots t_kx)\notin\mathbf Y$ because $\mathbf X\wedge \mathbf V=\mathbf Y\wedge \mathbf V$. 
In this case, it is easy to deduce from Lemmas~\ref{L: M(W) in V} and~\ref{L: id M(xy)} that $\mathbf Y$ satisfies a non-trivial identity $xt_1x\cdots t_kx\approx x^{f_0}t_1x^{f_1}\cdots t_kx^{f_k}=:\mathbf w^\prime$, where at least one of the numbers $f_0,\dots,f_k$ exceeds $1$.
Since $\mathbf X\subset\mathbf Y$, the identity $\mathbf w\approx \mathbf w^\prime$ holds in $\mathbf X$.
Clearly, this identity is satisfied by $\mathbf V$ as well.
Hence $\mathbf X\vee\mathbf V=\mathbf Y\vee \mathbf V$ satisfies $\mathbf w\approx \mathbf w^\prime$ and, therefore, $\mathbf Y$ satisfies $xt_1x\cdots t_kx\approx \mathbf w$.

Suppose now that $M(xt_1x\cdots t_kx)\notin\mathbf V$.
In this case, it is easy to deduce from Lemmas~\ref{L: M(W) in V} and~\ref{L: id M(xy)} that $\mathbf V$ satisfies the identity $xt_1x\cdots t_kx\approx \mathbf w$.
Since $\mathbf X\vee\mathbf V=\mathbf Y\vee \mathbf V$, we see that this identity holds in $\mathbf Y$. 

According to Claim~\ref{C: X=Y}, all the identities in $\Sigma$ are satisfied by $\mathbf Y$, contradicting $\mathbf X\subset \mathbf Y$. 
Therefore, $\mathbf V$ is a modular element of  of the lattice $\Mon$.
Theorem~\ref{T: modular} is thus proved.\qed


\small

\begin{thebibliography}{99}
\bibitem{Burris-Sankappanavar-81}
Burris, S., Sankappanavar, H.P.: A Course in Universal Algebra. Graduate Texts in Mathematics \textbf{78}, Springer-Verlag, Berlin---Heidelberg---New York (1981)

\bibitem{Gusev-23}
Gusev, S.V.: Varieties of aperiodic monoids with central idempotents whose subvariety lattice is distributive. Monatsh. Math. \textbf{201}, 79--108 (2023)

\bibitem{Gusev-Lee-21}
Gusev, S.V., Lee, E.W.H.: Cancellable elements of the lattice of monoid varieties. Acta Math. Hungar. \textbf{165}, 156--168 (2021)

\bibitem{Gusev-Lee-Vernikov-22}
Gusev, S.V., Lee, E.W.H., Vernikov, B.M.: The lattice of varieties of monoids. Japan. J. Math. \textbf{17}, 117--183 (2022)

\bibitem{Gusev-Vernikov-21}
Gusev, S.V., Vernikov, B.M.: Two weaker variants of congruence permutability for monoid varieties. Semigroup Forum \textbf{103}, 106--152 (2021)

\bibitem{Jackson-05}
Jackson, M.: Finiteness properties of varieties and the restriction to finite algebras. Semigroup Forum \textbf{70}, 159--187 (2005)

\bibitem{Jezek-81}
Je\v{z}ek, J.: The lattice of equational theories. Part~I: modular elements. Czechosl. Math. J. \textbf{31}, 127--152 (1981)

\bibitem{Jezek-McKenzie-93}
Je\v{z}ek, J., McKenzie, R.N.: Definability in the lattice of equational theories of semigroups. Semigroup Forum \textbf{46}, 199--245 (1993)

\bibitem{Vernikov-15}
Vernikov, B.M.: Special elements in lattices of semigroup varieties. Acta Sci. Math. (Szeged) \textbf{81}, 79--109 (2015)
\end{thebibliography}
\end{document}